\documentclass{amsart}

\usepackage[usenames,dvipsnames,svgnames,table]{xcolor}   
   
\newtheorem{thm}{Theorem}

\newcommand{\rd}{{\mathrm{d}}}

\newcommand{\C}{{\mathbb{C}}} 
\newcommand{\N}{{\mathbb{N}}} 
\newcommand{\NN}{{\mathbb{N}}} 

\DeclareSymbolFont{bbold}{U}{bbold}{m}{n}
\DeclareSymbolFontAlphabet{\mathbbold}{bbold}

\newcommand{\calG}{{\mathcal{G}}}
\newcommand{\calH}{{\mathcal{H}}}

\newcommand{\calS}{{\mathcal{S}}}

\newcommand{\ee}{{\rm e}}

\newcommand{\abs}[1]{\left\vert#1\right\vert}

\newcommand{\NOR}{{\mathrm{NOR}}}
\newcommand{\ABS}{{\mathrm{ABS}}}

\allowdisplaybreaks

\begin{document}   
        \title{Notes on Tractability Conditions\\   
         for Linear Multivariate Problems}   
    \author{Peter Kritzer and Henryk Wo\'zniakowski}   
        \date{\today}   
        \maketitle   
   
\begin{abstract}   
   
We study approximations of compact   
linear multivariate operators defined over Hilbert   
spaces. We provide necessary and sufficient conditions on various   
notions of tractability. These conditions are mainly given in terms of sums    
of certain functions depending on the singular values of the   
multivariate problem. They do not require the ordering of these   
singular values which in many cases is difficult to achieve.

\end{abstract}

\section{Introduction}   
   
Tractability of multivariate problems has become a popular research   
subject in the last 25 years. In this paper we study tractability in   
the worst case setting and for algorithms that use finitely many   
arbitrary continuous linear functionals.    
The information complexity of a   
$d$-variate compact linear operator   
$S_d$ is defined as the minimal number   
$n(\varepsilon,S_d)$ of such   
linear functionals which is needed to find an $\varepsilon$ approximation.     
There are various notions of   
tractability which may be summarized by the algebraic and exponential    
cases. For the algebraic case, we want to verify that the   
information complexity $n(\varepsilon,S_d)$ is bounded by certain   
functions of $d$ and $\max(1,\varepsilon^{-1})$ which are, in particular,   
polynomial or \emph{not}    
exponential in some powers of $d$ and $\max(1,\varepsilon^{-1})$.   
For the exponential case, we replace the pair   
$(d,\max(1,\varepsilon^{-1}))$   
by $(d,1+\ln\,\max(1,\varepsilon^{-1}))$,    
and consider the same notions of   
tractability as before.    
   
The algebraic case has been studied in many papers, and necessary and   
sufficient conditions on various notions of tractability are known   
in terms of sums of the singular values of $S_d$. The exponential case has   
been studied in a relatively small number of papers, and the   
corresponding necessary and sufficient conditions on tractability   
are provided in this paper.     
   
The information complexity requires to order the singular values of   
$S_d$. This is usually a difficult combinatorial problem. This problem   
is eliminated by the necessary and sufficient conditions on the   
singular values since they are given by sums which are invariant with   
respect to the ordering of the singular values.    
   
For the reader's convenience we provide all conditions    
for both algebraic and exponential cases for such notions of   
tractability as strong polynomial, polynomial, quasi-polynomial,    
various weak tractabilities, and uniform weak tractability. Furthermore,   
we do this for the absolute and normalized error criteria.      
The results are presented in five tables.    
   
In this paper   
we study general compact linear multivariate problems.   
In the next paper we illustrate the    
results of this paper for tensor product problems for which    
the singular values of a $d$-variate problem are given as products of   
the singular values of    
univariate problems.

\section{Preliminaries}   
   
Consider two sequences of Hilbert spaces    
$\{\calH_d\}_{d \in \N}$ and $\{\calG_d\}_{d \in \N}$,    
        and a sequence of compact    
linear solution operators    
$$   
\calS\,=\, \{S_d: \calH_d \to \calG_d\}_{d \in \N}.     
$$   
Here, we denote by $\N$ the set of positive integers,   
whereas $\N_0=\N\cup\{0\}$.

Our aim is to determine    
tractability conditions of the problem of    
        finding approximations to $\{S_d(f)\}$ for $f$ from the   
        unit ball of $\calH_d$. The approximations are obtained by algorithms   
$\{A_{d,n}: \calH_d \to \calG_d\}_{d \in \N,n\in\N_0}$.    
For $n=0$,    
we set $A_{d,0}:=0$, and for $n\ge 1$, $A_{d,n}(f)$    
depends only on $n$ continuous    
linear functionals $L_1(f), L_2(f),\ldots, L_n(f)$, i.e.,    
$$   
A_{d,n}(f)=\phi_n(L_1(f),L_2(f),\dots,L_n(f))   
$$   
with $\phi_n:\C^n\to \calG_d$ and $L_j\in \calH_d^*$.    
The choice of $L_j$ as well as $n$ can be adaptive, i.e.,   
$L_j=L_j(\cdot;L_1(f),L_2(f),\dots,L_{j-1}(f))$ and $n$ can be a   
function of the $L_j(f)$'s,    
see \cite{TWW88} as well as \cite{NW08} for details.    
We consider the worst case setting in which the error of $A_{d,n}$ is   
given by   
$$   
e(A_{d,n})\,=\,\sup_{\substack{f\in\calH_d \\ \|f\|_{\calH_d}\le1}}   
\|S_d(f)-A_{d,n}(f)\|_{\calG_d}.   
$$   
Let   
$$   
e(n,S_d)\,=\,\inf_{A_{d,n}}\,e(A_{d,n})   
$$   
denote the $n$th minimal worst case error,    
where the infimum is extended over all admissible algorithms $A_{d,n}$. Then    
the information complexity $n(\varepsilon,\calS_d)$ is   
the minimal number $n$ of continuous linear functionals    
which is needed to find an algorithm $A_{d,n}$ which approximates   
$S_d$ with error at most $\varepsilon$. More precisely,    
we consider the absolute (ABS) and normalized (NOR) error criteria   
in which    
\begin{eqnarray*}   
 n(\varepsilon,S_d)=n_{\ABS}(\varepsilon, S_d)&=&   
\min\{n\,\colon\,e(n,S_d)\le \varepsilon\},\\   
 n(\varepsilon,S_d)=n_{\NOR}(\varepsilon,S_d)&=&   
\min\{n\,\colon\,e(n,S_d) \le \varepsilon\,\|S_d\|\}.   
\end{eqnarray*}   
It is known from \cite{TWW88}, see also \cite{NW08},   
that the information complexity is fully determined by    
the singular values of $S_d$, which    
are the same as the square roots of the eigenvalues of    
the compact self-adjoint and positive semi-definite   
linear   
operator $W_d=S_d^\ast S_d:\calH_d\rightarrow\calH_d$.    
We denote these eigenvalues by $\lambda_{d,1}, \lambda_{d,2},\ldots$,    
ordered in a non-increasing fashion. Then    
for $\varepsilon>0$,   
\begin{eqnarray}   
n_{\ABS}(\varepsilon, S_d)   
&=&\min\{n\,\colon\,\lambda_{d,n+1}\le   
\varepsilon^2\},   
\label{eq:infcomplabs}\\   
n_{\NOR}(\varepsilon, S_d)&=&   
\min\{n\,\colon\,\lambda_{d,n+1}\le \varepsilon^2 \lambda_{d,1}\}.   
\label{eq:infcomplnor}   
\end{eqnarray}   
Clearly,    
$n_{\ABS}(\varepsilon, S_d)=0$ for $\varepsilon\ge   
\sqrt{\lambda_{d,1}}=\|S_d\|$, and    
$n_{\NOR}(\varepsilon, S_d)=0$ for $\varepsilon\ge1$.   
Therefore for $\ABS$ we can restrict ourselves to $\varepsilon\in(0,\|S_d\|)$,   
whereas for $\NOR$ to $\varepsilon\in(0,1)$.   
Since $\|S_d\|$ can be arbitrarily large, to deal simultaneously    
with $\ABS$ and $\NOR$ we consider $\varepsilon\in(0,\infty)$.    
It is known that $n_{\ABS/\NOR}(\varepsilon,S_d)$ is finite for all   
$\varepsilon>0$ iff $S_d$ is compact, which justifies our assumption   
about the compactness of~$S_d$.

We study how $n(\varepsilon, S_d)$    
depends on $\varepsilon$ and $d$. We compare two types of tractability:   
\begin{itemize}   
 \item Tractability with respect to $(d,\max(1,\varepsilon^{-1}))$   
 which is called algebraic tractability and abbreviated by ALG.   
 \item Tractability with respect to   
 $(d,1+\ln\,\max(1,\varepsilon^{-1}))$    
which is called exponential tractability and  abbreviated by EXP.    
\end{itemize}   
We now recall various notions of tractability which will be studied in   
this paper.   
\begin{itemize}   
\item {\bf{$\calS$ is ALG-SPT-ABS/NOR}} (strongly polynomially   
tractable in the algebraic case for the absolute or normalized error   
criterion) iff there are non-negative $C$ and $p$ such that   
for all \ $d\in\N,\ \varepsilon>0$ we have   
$$   
n_{\ABS/\NOR}(\varepsilon,S_d)\le C\,(\max(1,\varepsilon^{-1}))^p.   
$$   
The infimum of $p$ satisfying the bound above is denoted by $p^*$   
and called the exponent of ALG-SPT-ABS/NOR.    
\newline \qquad   
\item   
{\bf{$\calS$ is EXP-SPT-ABS/NOR}} (strongly polynomially   
tractable in the exponential case for the absolute or normalized error   
criterion) iff there are non-negative $C$ and $p$ such that   
for all\ $d\in\N,\ \varepsilon>0$ we have    
$$   
n_{\ABS/\NOR}(\varepsilon,S_d)\le C\,(   
1+\ln\,\max(1,\varepsilon^{-1}))^p.   
$$   
The infimum of $p$ satisfying the bound above is denoted by $p^*$   
and called the exponent of EXP-SPT-ABS/NOR.    
\newline \qquad   
\item    
{\bf{$\calS$ is ALG-PT-ABS/NOR}} (polynomially   
tractable in the algebraic case for the absolute or normalized error   
criterion) iff there are non-negative $C,p$, and $q$ such that   
for all \ $d\in\N,\ \varepsilon>0$ we have   
$$   
n_{\ABS/\NOR}(\varepsilon,S_d)\le C\,d^{\,q}\,   
(\max(1,\varepsilon^{-1}))^p.   
$$   
\vskip 0.5pc   
\item   
{\bf{$\calS$ is EXP-PT-ABS/NOR}} (polynomially   
tractable in the exponential case for the absolute or normalized error   
criterion) iff there are non-negative $C,p$, and $q$ such that   
for all\ $d\in\N,\ \varepsilon>0$ we have    
$$   
n_{\ABS/\NOR}(\varepsilon,S_d)\le C\,d^{\,q}\,   
(1+\ln\,\max(1,\varepsilon^{-1}))^p.   
$$   
\vskip 0.5pc   
\item   
{\bf{$\calS$ is ALG-QPT-ABS/NOR}} (quasi-polynomially   
tractable in the algebraic case for the absolute or normalized error   
criterion) iff there are non-negative $C$ and $p$ such that   
for all \ $d\in\N,\ \varepsilon>0$ we have    
$$   
n_{\ABS/\NOR}(\varepsilon,S_d)\le   
C\,\exp\left(p\,(1+\ln\,d)(1+\ln\,\max(1,\varepsilon^{-1}))\right).   
$$    
The infimum of $p$ satisfying the bound above is denoted by $p^*$   
and called the exponent of ALG-QPT-ABS/NOR.    
\newline \qquad   
\item   
{\bf{$\calS$ is EXP-QPT-ABS/NOR}} (quasi-polynomially   
tractable in the exponential case for the absolute or normalized error   
criterion) iff there are non-negative $C$ and $p$ such that   
for all \ $d\in\N,\ \varepsilon>0$ we have    
$$   
n_{\ABS/\NOR}(\varepsilon,S_d)\le   
C\,\exp\left(p\,(1+\ln\,d)(1+\ln(1+\ln\,\max(1,\varepsilon^{-1})))\right).   
$$    
The infimum of $p$ satisfying the bound above is denoted by $p^*$   
and called the exponent of EXP-QPT-ABS/NOR.    
\newline \qquad   
\item   
{\bf{$\calS$ is ALG-$(s,t)$-WT-ABS/NOR}} ($(s,t)$-weakly    
tractable in the algebraic case for the absolute or normalized error   
criterion) for positive $s$ and $t$ iff    
$$   
\lim_{d+\varepsilon^{-1}\to\infty}\   
\frac{\ln\, \max(1,n_{\ABS/\NOR}(\varepsilon,S_d))}   
{d^{\,t}+(\max(1,\varepsilon^{-1}))^s}\,=\,0.   
$$    
\newline \qquad   
\item   
{\bf{$\calS$ is EXP-$(s,t)$-WT-ABS/NOR}} ($(s,t)$-weakly    
tractable in the exponential case for the absolute or normalized error   
criterion) for positive $s$ and $t$ iff    
$$   
\lim_{d+\varepsilon^{-1}\to\infty}\   
\frac{\ln\,\max(1,n_{\ABS/\NOR}(\varepsilon,S_d))}{d^{\,t}+   
(1+\ln\,\max(1,\varepsilon^{-1}))^{s}}\,=\,0.   
$$    
\vskip 0.5pc   
\item    
{\bf{$\calS$ is ALG-UWT-ABS/NOR}} (uniformly weakly tractable   
in the algebraic case for the absolute or normalized error   
criterion) iff    
{\bf{$\calS$ is ALG-$(s,t)$-WT-ABS/NOR}} for all positive $s$ and $t$.   
\newline \qquad   
\item    
{\bf{$\calS$ is EXP-UWT-ABS/NOR}} (uniformly weakly tractable   
in the exponential case for the absolute or normalized error   
criterion) iff    
{\bf{$\calS$ is EXP-$(s,t)$-WT-ABS/NOR}} for all positive $s$ and $t$.   
\newline \qquad   
\end{itemize}   
   
For the algebraic case, necessary and sufficient   
conditions on the eigenvalues $\lambda_{d,n}$'s of $W_d$    
for various notions of tractability as well as    
the formulas for the exponents of tractability    
can be found in \cite{NW08}--\cite{NW12} for    
ALG-SPT, ALG-PT, ALG-QPT, and in \cite{WW17} for ALG-$(s,t)$-WT.   
ALG-UWT was defined in \cite{S13}, and conditions on tractability    
in this case can be easily obtained by combining conditions on   
ALG-$(s,t)$-WT as will be done in this paper.    
For the exponential case, corresponding necessary   
and sufficient conditions on $\lambda_{d,n}$'s as well as    
the formulas and bounds for the exponents of tractability will be derived in    
this paper.    
   
A few words of comment on these tractability definitions are in order.   
Note that the tractability notions are defined in terms of    
$\max(1,\varepsilon^{-1})$ and $1+\ln\,\max(1,\varepsilon^{-1})$.   
Before, this was   
usually done in terms of $\varepsilon^{-1}$ and $\ln\,\varepsilon^{-1}$ with   
an extra assumption that $\varepsilon\in(0,1)$. Since we want to   
consider arbitrary positive $\varepsilon$, the term   
$\varepsilon^{-1}$ is arbitrarily small for large $\varepsilon$, and   
then the term $\ln\,\varepsilon^{-1}$ is arbitrarily close to $-\infty$. These   
undesired properties  disappear if we consider   
$\max(1,\varepsilon^{-1})$   
instead of $\varepsilon^{-1}$, and    
$1+\ln\,\max(1,\varepsilon^{-1})$    
instead of   
$\ln\,\varepsilon^{-1}$, and they tend to $1$ as $\varepsilon$ becomes   
large.

We stress that we did not define the exponents of polynomial   
tractability. The reason is that in this case the pair $(p,q)$   
is usually \emph{not} uniquely defined and we may decrease, say, $p$   
at the expense of $q$ and vice versa. Obviously,    
we would be interested in finding the smallest possible $p$ and $q$   
for a given problem $\calS$.    
    
Modulo UWT, we listed the tractability notions from the most demanding   
to the most lenient ones. Obviously, we have    
\begin{eqnarray*}   
\mbox{ALG/EXP-SPT-ABS/NOR}&\implies&    
\mbox{ALG/EXP-\ PT-ABS/NOR}\ \implies \\   
\mbox{ALG/EXP-QPT-ABS/NOR}&\implies&     
\mbox{ALG/EXP-$(s,t)$-WT-ABS/NOR  $\forall\,s,t>0$.}   
\end{eqnarray*}   
Furthermore, for all $s_1\ge s_2$ and $t_1\ge t_2$   
$$   
 \mbox{ALG/EXP-$(s_2,t_2)$-WT-ABS/NOR}\ \ \implies\ \   
 \mbox{ALG/EXP-$(s_1,t_1)$-WT-ABS/NOR}.   
$$   
   
\section{Overview of previous and new results}   
We summarize previous and newly found conditions for the various    
tractability notions in the Tables \ref{table:spt}--\ref{table:uwt}.   
\begin{table}[h!]   
\begin{center}   
\caption{\quad {\bf{SPT}}}   
\label{table:spt}   
\begin{tabular}{|p{11cm}|}   
\hline \hline   
\\    
 \multicolumn{1}{|c|}{\textbf{$\calS$ is ALG-SPT-ABS iff}}\\   
\\   
\ \ \ $\exists\ \tau>0$ and   
 $\widetilde{C}\in\NN$ \ such that \\    
$$    
\sup_{d\in\NN} \, \sum_{j=\widetilde{C}}^\infty \lambda_{d,j}^\tau<\infty.   
$$\\    
The exponent $p^*=   
\inf\{2\,\tau:\ \mbox{$\tau$   
satisfies the bound above}\}$.   
\\    
 \hline \hline   
\\   
 \multicolumn{1}{|c|}{\textbf{$\calS$ is EXP-SPT-ABS iff}}\\   
\\   
\ \ \  $\exists\ \tau>0$ and   
 $\widetilde{C}\in\NN$ \ such that \\    
 $$   
\sup_{d\in\NN} \,    
\sum_{j=\widetilde{C}}^\infty \lambda_{d,j}^{j^{-\tau}}<\infty.   
$$\\    
The exponent $p^*=   
\inf\{1/\tau:\ \mbox{$\tau$   
satisfies the bound above}\}$.\\    
\hline \hline   
 \\   
\multicolumn{1}{|c|}{\textbf{$\calS$ is ALG-SPT-NOR iff}}\\   
\\   
\ \ \ $\exists\ \tau>0$ \ such that \\   
 $$   
\sup_{d\in\NN} \ \sum_{j=1}^\infty \left(\frac{\lambda_{d,j}}   
{\lambda_{d,1}}\right)^\tau<\infty.   
$$ \\     
The exponent $p^*=   
\inf\{2\,\tau:\ \mbox{$\tau$   
satisfies the bound above}\}$.\\    
   
 \hline \hline   
 \\   
\multicolumn{1}{|c|}{\textbf{$\calS$ is EXP-SPT-NOR iff}}\\   
\\   
\ \ \ $\exists\ \tau>0$ \ such that \\    
 $$   
\sup_{d\in\NN}\   
 \sum_{j=1}^\infty \left(\frac{\lambda_{d,j}}   
{\lambda_{d,1}}\right)^{j^{-\tau}}<\infty.   
$$   
\\    
The exponent $p^*=   
\inf\{1/\tau:\ \mbox{$\tau$   
satisfies the bound above}\}$.\\    
\hline \hline   
\end{tabular}   
\end{center}   
\end{table}   
\newpage   
   
We stress that for SPT-ABS the values of finitely many largest    
eigenvalues do not matter and they may be arbitrarily large.    
For SPT-NOR, the eigenvalues are normalized and their quotients    
are at most $1$.   
However, the multiplicity of the largest eigenvalue must be    
uniformly bounded in $d$ to achieve SPT.

   
\begin{table}[h!]   
\begin{center}   
\caption{{\bf{PT}}}    
\label{table:pt}   
\begin{tabular}{|p{11cm}|}   
\hline\hline   
 \\   
\multicolumn{1}{|c|}{\textbf{$\calS$ is ALG-PT-ABS iff}}\\   
\\   
\ \ \ $\exists$ \    
$\tau_1,\tau_3 \ge 0$ and $\tau_2,\,\widetilde{C}>0$ such that\\    
 $$   
\sup_{d\in\NN}\ d^{-\tau_1}    
\sum_{j=\lceil \widetilde{C} d^{\tau_3}\rceil}^\infty    
\lambda_{d,j}^{\tau_2}<\infty.   
$$   
\\    
 \hline \hline   
 \\   
\multicolumn{1}{|c|}{\textbf{$\calS$ is EXP-PT-ABS iff}}\\   
\\   
\ \ \ $\exists$\    
 $\tau_1,\tau_3 \ge 0$ and $\tau_2,\,\widetilde{C}> 0$ such that \\   
$$ \sup_{d\in\NN}\ d^{-\tau_1} \sum_{j=\lceil \widetilde{C}   
 d^{\tau_3}\rceil}^\infty \lambda_{d,j}^{j^{-\tau_2}}<\infty.   
$$ \\   
 \hline \hline   
 \\   
\multicolumn{1}{|c|}{\textbf{$\calS$ is ALG-PT-NOR iff}}\\   
\\   
\ \ \ $\exists$\    
$\tau_1\ge0$ and $\tau_2>0$ such that\\    
 $$   
\sup_{d\in\NN}\ d^{-\tau_1}   
\sum_{j=1}^\infty   
\left(\frac{\lambda_{d,j}}{\lambda_{d,1}}\right)^{\tau_2}   
<\infty.   
$$   
\\   
 \hline \hline   
 \\   
\multicolumn{1}{|c|}{\textbf{$\calS$ is EXP-PT-NOR iff}}\\   
\\   
\ \ \ $\exists$\    
 $\tau_1\ge 0$ and $\tau_2>0$ such that\\ \\    
$$   
\sup_{d\in\NN}\    
d^{-\tau_1} \sum_{j=1}^\infty   
\left(\frac{\lambda_{d,j}}{\lambda_{d,1}}   
\right)^{j^{-\tau_2}}<\infty.   
$$ \\   
 \hline \hline   
\end{tabular}   
\end{center}   
\end{table}   
   
We stress that for PT-ABS the values of polynomially many largest   
eigenvalues are irrelevant. Again, for PT-NOR all of them matter   
and the multiplicity of the largest eigenvalue must be polynomially   
bounded in $d$.     
   
The only difference between SPT and PT is that   
the corresponding sums of some powers of the eigenvalues must be   
bounded in the SPT case whereas in the PT case they may polynomially   
increase with $d$.

\newpage   
   
\begin{table}[h!]   
\begin{center}   
\caption{{\bf{QPT}}}   
\label{table:qpt}   
\begin{tabular}{|p{11cm}|}   
\hline\hline   
 \\   
\multicolumn{1}{|c|}{\textbf{$\calS$ is ALG-QPT-ABS iff}}\\   
\\   
\ \ \ $\exists$\    
 $\tau_1\ge 0$, and $\tau_2,\,\widetilde{C}>0$ such that \\    
  $$   
\sup_{d\in\NN}\ d^{-2}\left(\sum_{j=\lceil \widetilde{C} d^{\tau_1}   
\rceil}^\infty \lambda_{d,j}^{\tau_2 (1+\ln\,d)}\right)^{1/\tau_2}   
<\infty.   
$$ \\    
The exponent $p^*=   
\inf\{\,\max(\tau_1,2\tau_2):\    
\mbox{$\tau_1,\tau_2$ satisfy the bound above}\}$.   
\vskip 0.5pc   
\\   
 \hline \hline   
 \\   
\multicolumn{1}{|c|}{\textbf{$\calS$ is EXP-QPT-ABS iff}}\\   
\\   
\ \ \ $\exists$\    
$\tau>0$ such that\\    
$$   
 \sup_{d\in\NN}\ d^{-\tau}    
\sum_{j=1}^\infty   
  \left[1+\tfrac12\,\ln\,\max\left(1,\frac{1}{\lambda_{d,j}}\right)   
\right]^{-\tau(1+\ln\,d)}<\infty.   
$$   
\\   
The exponent $p^*=   
\inf\{\,\tau:\,    
\mbox{$\tau$ satisfies the bound above}\}$.   
\vskip 0.5pc   
\\   
\hline \hline   
\\   
\multicolumn{1}{|c|}{\textbf{$\calS$ is ALG-QPT-NOR iff}}\\   
\\   
\ \ \ $\exists$\    
 $\tau >0$ such that \\    
 $$   
\sup_{d\in\NN}\, d^{-2}\left(\sum_{j=1}^\infty    
\left(\frac{\lambda_{d,j}}{\lambda_{d,1}}   
\right)^{\tau (1+\ln\,d)}\right)^{1/\tau}<\infty.   
$$ \\   
The exponent $p^*=   
\inf\{\,2\tau:\, \    
\mbox{$\tau$ satisfies the bound above}\}$.   
\vskip 0.5pc   
\\   
 \hline \hline   
 \\   
\multicolumn{1}{|c|}{\textbf{$\calS$ is EXP-QPT-NOR iff}}\\   
\\   
\ \ \ $\exists$\    
 $\tau> 0$ such that\\ \\    
 $$   
\sup_{d\in\NN}\ d^{-\tau} \sum_{j=1}^\infty   
  \left[1+\tfrac12\ln\,\frac{\lambda_{d,1}}   
{\lambda_{d,j}}\right]^{-\tau(1+\ln\,d)}<\infty.   
$$ \\   
The exponent $p^*=   
\inf\{\,\tau:\, \    
\mbox{$\tau$ satisfies the bound above}\}$.   
\vskip 0.5pc   
\\   
\hline \hline   
\end{tabular}   
\end{center}   
\end{table}   
   
\newpage    
   
\begin{table}[h!]   
\begin{center}   
\caption{{\bf{$(s,t)$-WT}}}   
\label{table:wt}   
\begin{tabular}{|p{11cm}|}   
\hline \hline   
 \\   
\multicolumn{1}{|c|}{\textbf{$\calS$ is ALG-$(s,t)$-WT-ABS iff}}\\   
\\   
$$   
\sup\limits_{d\in\NN}\ \exp (-cd^{\,t})\sum\limits_{j=1}^\infty    
 \exp \left(-c\left(\frac1{\lambda_{d,j}}\right)^{s/2}\right)<\infty\ \ \    
\mbox{for all $c>0$.}   
$$    
\\   
 \hline \hline   
 \\   
\multicolumn{1}{|c|}{\textbf{$\calS$ is EXP-$(s,t)$-WT-ABS iff}}\\   
\\   
$$   
\sup\limits_{d\in\NN}\ \exp \left(-cd^{\,t}\right)\,\sum\limits_{j=1}^\infty    
 \exp \left(-c   
 \left[1+\ln\left(2\,\max\left(1,\frac1{\lambda_{d,j}}\right)\right)   
\right]^s   
 \right)<\infty   
\ \ \ \mbox{for all $c>0$.}   
$$ \\   
 \hline\hline   
 \\   
\multicolumn{1}{|c|}{\textbf{$\calS$ is ALG-$(s,t)$-WT-NOR iff}}\\   
\\   
$$   
\sup\limits_{d\in\NN}\ \exp\left(-cd^{\,t}\right)\,\sum\limits_{j=1}^\infty    
 \exp   
 \left(-c\left(\frac{\lambda_{d,1}}{\lambda_{d,j}}\right)^{s/2}\right)<\infty   
\ \ \  \mbox{for all $c>0$.}   
$$  \\   
 \hline \hline   
 \\   
\multicolumn{1}{|c|}{\textbf{$\calS$ is EXP-$(s,t)$-WT-NOR iff}}\\   
\\   
$$   
\sup\limits_{d\in\NN}\ \exp\left(-cd^{\,t}\right)\,   
\sum\limits_{j=1}^\infty    
 \exp \left(-c \left[1+\ln   
     \frac{2\lambda_{d,1}}{\lambda_{d,j}}\right]^s \right)<\infty   
\ \ \  \mbox{for all $c>0$.}   
$$    
\\   
 \hline \hline   
\end{tabular}   
\end{center}   
\end{table}   
For the case of ALG, we need to guarantee the convergence of the   
series depending on $\lambda_{d,j}^{-s/2}$ or   
$(\lambda_{d,1}/\lambda_{d,j})^{s/2}$,    
whereas for the case of EXP,   
the corresponding series now depends on the logarithms of   
$\lambda_{d,j}^{-1}$ or   
$(\lambda_{d,1}/\lambda_{d,j})$   
raised to the power~$s$.   
Furthermore, in both cases, the convergent series for a fixed $d$ must   
be at most of order $\exp(cd^{\,t})$ and this must hold for all   
positive $c$.   
   
Note that for ABS, the number of eigenvalues $\lambda_{d,j}\ge1$    
must be of order $\exp(o(d^{\,t}))$, whereas for NOR,   
the multiplicity of the largest eigenvalues $\lambda_{d,1}$   
must be of order $\exp(o(d^{\,t}))$.    
      
\newpage   
   
\begin{table}[h!]   
\begin{center}   
\caption{{\bf{UWT}}}   
`\label{table:uwt}   
\begin{tabular}{|p{11cm}|}   
\hline \hline   
 \\   
\multicolumn{1}{|c|}{\textbf{$\calS$ is ALG-UWT-ABS iff}}\\   
\\   
$$   
\lim_{n\to\infty}\ \ \inf_{d\in\N:\ d\le [\ln\,n]^k}\    
\frac{\ln\,\frac{1}{\lambda_{d,n}}}{\ln\,\ln\,n}\ =\ \infty \ \ \    
\mbox{for all $k\in\N$},   
$$    
\\   
 \hline \hline   
 \\   
\multicolumn{1}{|c|}{\textbf{$\calS$ is EXP-UWT-ABS iff}}\\   
\\   
$$   
\lim_{n\to\infty}\ \ \inf_{d\in\N:\ d\le [\ln\,n]^k}\   
\frac{\ln \left(\max\left(1,\ln\,\frac{1}{\lambda_{d,n}}\right)\right)}{\ln\,\ln\,n}\ =   
\ \infty \ \ \    
\mbox{for all $k\in\N$}.   
$$    
\\   
 \hline \hline   
\\   
\multicolumn{1}{|c|}{\textbf{$\calS$ is ALG-UWT-NOR iff}}\\   
\\   
$$   
\lim_{n\to\infty}\ \ \inf_{d\in\N:\ d\le [\ln\,n]^k}\    
\frac{\ln\,\frac{\lambda_{d,1}}{\lambda_{d,n}}}{\ln\,\ln\,n}\ =\ \infty \ \ \    
\mbox{for all $k\in\N$},   
$$    
\\   
 \hline \hline   
\\    
 \multicolumn{1}{|c|}{\textbf{$\calS$ is EXP-UWT-NOR iff}}\\   
\\   
$$   
\lim_{n\to\infty}\ \ \inf_{d\in\N:\ d\le [\ln\,n]^k}\    
\frac{\ln\left(\max\left(1,\ln\,\frac{\lambda_{d,1}}{\lambda_{d,n}}\right)\right)}{\ln\,\ln\,n}\ =   
\ \infty \ \ \    
\mbox{for all $k\in\N$}.   
$$    
\\   
 \hline \hline   
\end{tabular}   
\end{center}   
\end{table}   
\vskip 1pc   
This is the only table which depends on the ordered eigenvalues   
$\lambda_{d,n}$. We obtain UWT if $\lambda_{d,n}$'s go to zero   
sufficiently fast. Note that the case of ALG requires the single   
logarithm of $1/\lambda_{d,n}$ or $\lambda_{d,1}/\lambda_{d,n}$,   
whereas the case of EXP requires the double logarithms of the same   
expressions. This quantifies how much harder the case of EXP is as    
compared to the case of ALG.    
   
For example,   
take $\lambda_{d,n}=n^{-\alpha}$ for an arbitrary $\alpha>0$    
for all $n,d\in\NN$. Then ABS=NOR and    
we obtain ALG-UWT-ABS/NOR, however EXP-UWT-ABS/NOR does not hold.     
Hence, polynomial decay of the eigenvalues $\lambda_{d,n}$ is    
enough for ALG-UWT-ABS/NOR, and not enough for EXP-UWT-ABS/NOR.   
On the other hand, we obtain EXP-UWT-ABS/NOR if, say,    
$\lambda_{d,n}=\exp(-n^{\alpha})$   
for an arbitrary $\alpha>0$ for all $n,d\in \NN$.     
   
The dependence on $d$ is only through the infimum of $d\le   
[\ln\,n]^k$. Note that for large $n$ or $k$,    
we need to consider more $d$'s and even the smallest quotient   
with respect to $d$    
must be sufficiently large for large $n$.     
   
\newpage   
\section{Proofs}   
In this section we are ready to    
prove necessary and sufficient   
conditions on the eigenvalues $\lambda_{d,j}$'s for    
tractability in the exponential case presented in the tables above.    
The subsequent subsections will   
address these conditions for various notions of tractability.   
   
It turns out that the proofs    
for the absolute and normalized error criteria are   
similar. Therefore we combine them by using the abbreviation   
$$   
 \mathrm{CRI}_d=\begin{cases}   
                 1 &\mbox{for $\ABS$},\\   
                 \lambda_{d,1} &\mbox{for $\NOR$}.   
                 \end{cases}   
$$   
   
\subsection{Strong Polynomial and Polynomial Tractability}

\begin{thm}[EXP-SPT/PT-ABS/NOR]\label{thm1}   
\quad   
   
\vskip 0.5pc   
$\calS$ is EXP-SPT/PT-ABS/NOR iff    
there exist $\tau_1,\tau_3\ge0$ and $\tau_2,\widetilde{C}>0$ such that   
\begin{equation}\label{eq:EXP-SPT-ABS}   
 M:=\sup_{d\in\NN}\,d^{\,-\tau_1}\,\sum_{j=\lceil \widetilde{C}\,   
d^{\,\tau_3}\rceil}^\infty    
\left(\frac{\lambda_{d,j}}{{\rm CRI}_d}\right)^{j^{-\tau_2}}<\infty.   
\end{equation}   
For SPT, we have $\tau_1=\tau_3=0$, and for NOR we have   
$\widetilde{C}=1$ and $\tau_3=0$.   
   
If this holds then   
\[   
 n_{\ABS/\NOR}(\varepsilon, S_d) \le    
\lfloor M\ee d^{\,\tau_1}\rfloor+    
 \lceil \widetilde{C}\,d^{\,\tau_3}\rceil+    
\lceil\max(0,2\ln \varepsilon^{-1})^{1/\tau_2}\rceil,   
\]   
and the exponent of EXP-SPT-ABS/NOR is    
$$   
p^*=\inf\{1/\tau_2:\ \mbox{$\tau_2$   
    satisfies \eqref{eq:EXP-SPT-ABS}}\}.   
$$   
\end{thm}   
\begin{proof}   
Let us first assume that \eqref{eq:EXP-SPT-ABS} holds. We then need to show   
that for some $C,q,p\ge0$ we have   
\[   
  n_{\ABS/\NOR}(\varepsilon, S_d) \le C\,d^{\,q}    
(1+ \ln\,\max(1,\varepsilon^{-1}))^p   
\ \ \ \mbox{for all $\varepsilon>0$ and $d\in\N$,}   
\]   
where $q=0$ in the case of SPT.   
To this end, let    
\[   
  B_d:=\left\{j\in\NN:\, j\ge \lceil \widetilde{C}\,d^{\,\tau_3}\rceil   
    \quad   
\mbox{and}\quad \left(\frac{\lambda_{d,j}}{{\rm CRI}_d}   
\right)^{j^{-\tau_2}}>\frac1{\ee}\right\}.    
\]   
Since \eqref{eq:EXP-SPT-ABS} holds,    
we see that $\abs{B_d}< M\ee\,d^{\,\tau_1}$ and    
$\abs{B_d}\le \lfloor M\ee\,d^{\,\tau_1}\rfloor$.    
   
Suppose now that $j\ge\lceil \widetilde{C}\,d^{\,\tau_3}\rceil$    
but  $j\not\in B_d$, which means that    
\[   
\left(\frac{ \lambda_{d,j}}{{\rm CRI}_d}\right)^{j^{-\tau_2}}\le\frac1{\ee},   
\quad\mbox{or equivalently,}\quad    
\frac{\lambda_{d,j}}{{\rm CRI}_d}\le\exp\left(-j^{\tau_2}\right).   
\]   
This implies that   
\begin{equation}\label{eq:SPTABSjlarge}   
 \lambda_{d,j}\le \varepsilon^2\,{\rm CRI}_d \quad \mbox{if}\quad    
j\notin B_d \ \ \mbox{and}\ \ j\ge   
\max\left(\lceil\widetilde{C}\,d^{\tau_3}\rceil   
,\lceil \max(0,    
2\ln \varepsilon^{-1})^{1/\tau_2}\rceil\right).   
\end{equation}   
   
Due to \eqref{eq:infcomplabs} and \eqref{eq:infcomplnor}, our observation    
regarding $\abs{B_d}$, and \eqref{eq:SPTABSjlarge}, it follows that   
\[   
 n_{\ABS/\NOR}(\varepsilon, S_d) \le \lfloor M\ee\,d^{\,\tau_1}\rfloor    
+ \lceil \widetilde{C}\,d^{\,\tau_3}\rceil +    
\lceil\max(0,2\ln \varepsilon^{-1})^{1/\tau_2}\rceil,   
\]   
as claimed. This easily implies    
\[   
  n_{\ABS/\NOR}(\varepsilon, S_d) \le C \,d^{\,\max(\tau_1,\tau_3)}\,   
(1+ \ln\,\max(1,\varepsilon^{-1}))^{1/\tau_2}   
 \]   
for some suitably chosen $C$. Hence, EXP-SPT/PT-ABS/NOR holds.   
   
For SPT, we have $\tau_1=\tau_3=0$, and $q=\max(\tau_1,\tau_3)=0$.    
For the exponent of SPT we have $p^*\le \inf\{1/\tau_2:\ \mbox{$\tau_2$    
satisfies \eqref{eq:EXP-SPT-ABS}}\}$.    
   
\vskip 1pc   
Let us now assume that there are non-negative $C,q$, and $p$ such that     
\[   
  n_{\ABS/\NOR}(\varepsilon, S_d) \le C\,d^{\,q}\,    
(1+ \ln\,\max(1,\varepsilon^{-1}))^p   
 \]   
holds for all $d\in\NN$ and all $\varepsilon>0$. For SPT we have   
$q=0$ and $p$ can be arbitrarily close to $p^*$, say $p=p^*+\delta$   
for   
some (small) positive $\delta$.    
   
 Then    
\[   
 \lambda_{d,n_{\ABS/\NOR}(\varepsilon, S_d) +1} \le \varepsilon^2\,   
{\rm CRI}_d.      
\]   
The latter inequality holds for all $\varepsilon>0$,    
but we will use it only for $\varepsilon \in(0,1]$.   
Without loss of generality we may assume that $C\ge1$.    
   
Since the eigenvalues $\lambda_{d,j}$ are non-increasing, we have    
\begin{equation}\label{eq:SPTABSlambda}   
 \lambda_{d,\lfloor C\,d^{\,q} (1+    
\ln\,\max(1,\varepsilon^{-1}))^p \rfloor +1}    
\le \varepsilon^2\,{\rm CRI}_d.     
\end{equation}   
Let   
\[   
 j=\lfloor C\,d^{\,q} (1+ \ln\,\max(1,\varepsilon^{-1}))^p \rfloor +1.   
\]   
If we vary $\varepsilon\in (0,1]$, we see    
that $j=j_d^*,j_d^*+1, j_d^*+2,\dots$, where    
$$   
j_d^*=\lfloor C\,d^{\,q}\rfloor+1\ge2.   
$$    
Note that   
\[   
j\le C\,d^{\,q}(1+\ln\,\max(1,\varepsilon^{-1}))^p+1,                      
\]   
or equivalently,   
\[   
\varepsilon   
\le \exp\left(-\frac{1}{C^{1/p}\,d^{\,q/p}}(j-1)^{1/p} +1\right).   
\]   
For $j\ge j_d^*\ge2$ we have $(j-1)\ge j/2$ and therefore   
$$   
\varepsilon\le    
\,\ee\,\exp\left(-\frac{1}{(2C)^{1/p}\,d^{\,q/p}}   
\,j^{1/p}\right).   
$$   
By inserting into \eqref{eq:SPTABSlambda}, we see that    
\begin{equation}\label{157157}   
\frac{\lambda_{d,j}}{{\rm CRI}_d}    
\le    
\ee^2\exp\left(-\frac{2}{(2C)^{1/p}\,d^{\,q/p}}\,j^{1/p}\right)   
\ \ \ \mbox{for all $j\ge j_d^*$}.   
\end{equation}   
Consequently,   
$$   
 \sum_{j=j_d^*}^\infty    
\left(\frac{\lambda_{d,j}}{{\rm CRI}_d}\right)^{j^{-\tau_2}} \le     
\ee^2\sum_{j=j_d^*}^\infty   
\exp\left(-\frac{2}{(2C)^{1/p}\,d^{\,q/p}}\,j^{1/p-\tau_2} \right).   
$$   
Choose $\tau_2 < 1/p$, or equivalently, $1/p -\tau_2 > 0$.   
Then the terms of the last sum are decreasing in $j$ and     
\begin{eqnarray*}   
\sum_{j=j_d^*}^\infty   
\exp\left(-\frac{2}{(2C)^{1/p}\,d^{\,q/p}}\,j^{1/p-\tau_2} \right)   
 &\le&   
\int_{j_d^*-1}^\infty    
\exp\left(-\frac{2}{(2C)^{1/p}\,d^{\,q/p}}\,x^{1/p-\tau_2} \right)\rd x\\   
&\le&\int_{0}^\infty    
\exp\left(-\frac{2}{(2C)^{1/p}\,d^{\,q/p}}\,x^{1/p-\tau_2} \right)\rd x.   
\end{eqnarray*}   
We now put   
$$   
B:= \frac{2}{(2C)^{1/p} d^{\,q/p}},    
\quad V:= 1/p-\tau_2,   
$$   
so the above integral equals   
\[   
 I:=\int_{0}^\infty \exp(-B\, x^V )\,\rd x.   
\]   
By substituting $t$ for $Bx^V$, we obtain    
\[   
I=B^{-1/V} \frac{1}{V} \int_{0}^\infty t^{1/V-1}\exp (-t) \rd t=    
B^{-1/V} \frac{1}{V} \, \Gamma \left(\frac{1}{V}\right),   
\]   
where   
\[\Gamma (s):=\int_0^\infty t^{s-1}\exp (-t) \rd t\]   
is the Gamma function.    
We therefore get   
$$   
 I=   
\frac{p}{1-\tau_2 p} 2^{\frac{1-p}{1-\tau_2 p}}\,    
C^{\frac{1}{1-\tau_2 p}}\,    
d^{\,\frac{q}{1-\tau_2 p}}\, \Gamma\left(\frac{p}{1-\tau_2 p}\right).   
$$

In summary,   
$$   
\sum_{j=j^*_d}^\infty    
\left(\frac{\lambda_{d,j}}{{\rm CRI}_d}\right)^{j^{-\tau_2}}    
=\mathcal{O}\left(d^{\,q/(1-\tau_2p)}\right)   
\ \ \ \mbox{with}\ \ \  j^*_d=\mathcal{O}(d^{\,q}),    
$$   
where the last two factors in the big $\mathcal{O}$ notation    
are independent of $d$.   
   
Consider now ABS.   
We see that \eqref{eq:EXP-SPT-ABS} holds for   
$\tau_1=q/(1-\tau_2p)$, $\tau_2<1/p$ and $\tau_3=q$.   
For SPT, we have $q=0$ which implies that $\tau_1=\tau_3=0$, and the   
exponent of SPT is    
$\inf\{1/\tau_2:\ \mbox{$\tau_2$ satisfies   
  \eqref{eq:EXP-SPT-ABS}}\}=p=p^*+\delta$.   
Since this holds for all positive $\delta$,    
together with the previous inequality we     
  conclude that $p^*=\inf\{1/\tau_2:\ \mbox{$\tau_2$ satisfies   
    \eqref{eq:EXP-SPT-ABS}\}}$, as claimed.    
   
Finally, for NOR we can   
  take $\widetilde{C}=1$ and $\tau_3=0$ and use the fact that   
$$   
\sum_{j=1}^\infty\left(\frac{\lambda_{d,j}}{\lambda_{d,1}}\right)^{j^{-\tau_2}}   
\le j^*_d+   
\sum_{j=j^*_d}^\infty   
\left(\frac{\lambda_{d,j}}{\lambda_{d,1}}\right)^{j^{-\tau_2}}   
=\mathcal{O}\left(d^{\,q}+d^{\,q/(1-\tau_2p)}\right),   
$$   
and \eqref{eq:EXP-SPT-ABS} holds with $\tau_1=q/(1-\tau_2p)$ for all    
$\tau_2<1/p$. The rest is done as for ABS.    
This completes the proof.   
\end{proof}   
   
   
   
\subsection{Quasi-polynomial tractability}   
   
\begin{thm}[EXP-QPT-ABS/NOR]\label{thm:exp-qpt-nor}   
\quad   
   
$\calS$ is EXP-QPT-ABS/NOR iff there exists   
$\tau>0$  such that    
\begin{equation}\label{eq:EXP-QPT-NOR}   
 M:=\sup_{d\in\NN}\ d^{-\tau}    
\sum_{j=1}^\infty   
\left[1+\tfrac12\,\ln\,\max\left(1,   
\frac{{\rm CRI}_d}{\lambda_{d,j}}   
\right)\right]^{-\tau(1+\ln\,d)} <\infty.   
\end{equation}   
If this holds then   
$$   
n_{\ABS/\NOR}(\varepsilon,S_d)\le 1+M\,d^{\tau}+   
M\,d^{\,\tau}\left[\max(0,1+\ln\,\varepsilon^{-1})   
\right]^{\tau(1+\ln\,d)},   
$$   
and the exponent of EXP-QPT-ABS/NOR is   
$$   
p^*=   
\inf\{\,\tau\,:\ \tau\ \ \mbox{satisfies \eqref{eq:EXP-QPT-NOR}}\}.   
$$    
\end{thm}   
   
\begin{proof}   
Let us first assume that \eqref{eq:EXP-QPT-NOR} holds. We then need to show   
that for some $C,p>0$ we have    
\[   
  n_{\ABS/\NOR}(\varepsilon, S_d) \le C    
\exp[\,p\,(1+\ln\,d)\,(1+\ln (1+    
\ln\,\max(1,\varepsilon^{-1})))]   
\]   
for all $\varepsilon>0$ and $d\in \NN$. Let   
$$   
j^*_1(d)=|\{\,j\in\NN\,:\ {\rm CRI}_d/\lambda_{d,j}<1\,\}|.   
$$   
Note that $j^*_1(d)=0$ for NOR, whereas $j^*_1(d)$ may be positive for   
ABS.   
   
{}From \eqref{eq:EXP-QPT-NOR} we conclude that   
$$   
M\ge d^{-\tau}\,\sum_{j=1}^{j^*_1(d)}1=d^{-\tau}\,j^*_1(d).   
$$    
Hence,   
$$   
j^*_1(d)\le M\,d^{\,\tau}\ \ \ \mbox{for both ABS and NOR}.   
$$   
For $j>j^*_1(d)$ we have ${\rm CRI}_d/\lambda_{d,j}\ge1$ and    
$$   
1+\tfrac12\,\ln\,\max\left(1,\ln\,\frac{{\rm CRI}_d}{\lambda_{d,j}}\right)   
=1+\tfrac12\,\ln\,\frac{{\rm CRI}_d}{\lambda_{d,j}}\ge1.   
$$   
Again due to \eqref{eq:EXP-QPT-NOR}, we have   
$$   
\sum_{j=j_1^*(d)+1}^\infty   
\left[1+\tfrac12\,\ln\,   
\frac{{\rm CRI_d}}{\lambda_{d,j}}\right]^{-\tau(1+\ln\,d)}   
\le M\,d^{\,\tau}.   
$$   
Note that the terms of the last sum are non-increasing. Therefore   
$$   
\left(n-j_1^*(d)\right)   
\left[1+\tfrac12\,\ln\,   
\frac{{\rm CRI_d}}{\lambda_{d,n}}\right]^{-\tau(1+\ln\,d)}   
\le M\,d^{\,\tau}.   
$$   
After simple algebraic manipulations we conclude that   
$$   
\sqrt{\frac{\lambda_{d,n}}{{\rm CRI_d}}}\le\ee\,   
\exp\left(-\left[\frac{n-j_1^*(d)}{M\,d^{\,\tau}}\right]^{1/(\tau(1+\ln\,d))}   
\right).   
$$   
We now assume that $\varepsilon\in(0,\ee)$.    
Hence, the right-hand side of the last inequality is at most   
$\varepsilon$ for    
$$   
n\ge   
j_1^*(d)+M\,d^{\,\tau}   
\left[1+\ln\,\varepsilon^{-1})   
\right]^{\tau(1+\ln\,d)}.   
$$   
Using the estimate for $j_1^*(d)$, this means that    
$$   
n:=n_{\ABS/\NOR}(\varepsilon,S_d)\le 1+M\,d^{\tau}+   
M\,d^{\,\tau}\left[1+\ln\,\varepsilon^{-1}   
\right]^{\tau(1+\ln\,d)},   
$$   
as claimed.    
   
This can be slightly overestimated by   
$$   
 n\le   
(1+M)\exp(\tau\,\ln\,d)+M   
\exp\left(\tau(1+\ln\,d)(1+\ln(1+\ln\,\varepsilon^{-1}))\right).    
$$   
It is easy to check that   
$$   
\ln\,\left(1+\ln\,\varepsilon^{-1}\right)   
\le 1+\ln(1+\ln\,\max(1,\varepsilon^{-1}))   
$$   
and therefore    
\[   
  n_{\ABS/\NOR}(\varepsilon, S_d) \le    
(1+2M)\,   
\exp\left(\tau(1+\ln\,d)(1+   
\ln(1+\ln\,\max(1,\varepsilon^{-1})))\right).   
 \]   
This means that EXP-QPT-ABS/NOR holds.      
Furthermore, the exponent of EXP-QPT-ABS/NOR is at most   
$\inf\{\,\tau\,: \ \mbox{$\tau$ satisfies    
\eqref{eq:EXP-QPT-NOR}}\,\}$.     
   
\bigskip   
   
Assume now that EXP-QPT-ABS/NOR holds, i.e., for some $C\ge1$ and   
positive~$p$ we have    
\[   
  n_{\ABS/\NOR}(\varepsilon, S_d) \le C    
\exp\left[p\,(1+\ln\,d)\,(1+\ln(1+ \ln\,\max(1,\varepsilon^{-1}))) \right]   
 \]   
holds for all $d\in\NN$ and all $\varepsilon>0$. This can be rewritten as   
\[   
 n_{\ABS/\NOR}(\varepsilon, S_d) \le C    
\ee^{\,p} d^{\,p} [1+\ln\,\max(1,\varepsilon^{-1})]^{p(1+\ln\,d)}.   
\]   
We have   
\[   
 \lambda_{d,n_{\ABS/\NOR}(\varepsilon, S_d) +1} \le    
\varepsilon^2 {\rm CRI}_d.    
\]   
Since the eigenvalues $\lambda_{d,j}$ are non-increasing, we have    
\begin{equation}\label{eq:QPTlambda}   
 \lambda_{d,\lfloor C \ee^{\,p} d^{\,p}    
[1+\ln\,\max(1,\varepsilon^{-1})]^{p(1+\ln\,d)}    
\rfloor +1} \le \varepsilon^2 {\rm CRI}_d.    
\end{equation}   
Although the estimate of $n_{\ABS/\NOR}(\varepsilon,S_d)$ holds for   
all $\varepsilon>0$, we assume that $\varepsilon\in(0,1]$.
   
If we vary $\varepsilon\in (0,1]$, we see that    
$$   
j=\lfloor C \ee^{\,p} d^{\,p}    
[1+\ln\,\varepsilon^{-1}]^{p(1+\ln\,d)} \rfloor +1   
$$    
attains the values    
$j=j_d,j_d+1,\dots$, where   
$$   
j_d=\left\lfloor C \ee^{\,p}   
d^{\,p}\right\rfloor+1\ge 2.   
$$   
Furthermore, we have   
\[   
 j\le C \ee^{\,p} d^{\,p}    
[1+\ln\,\varepsilon^{-1}]^{p(1+\ln\,d)}  +1   
\]   
or equivalently,   
\[   
\varepsilon    
\le \exp\left(-\left(\frac{j-1}{C \ee^{\,p} d^{\,p}}   
\right)^{1/(p(1+\ln\,d))} +1\right).   
\]   
Inserting this into \eqref{eq:QPTlambda} we conclude that   
$$   
1+\tfrac12\,\ln\,\frac{{\rm CRI}_d}{\lambda_{d,j}}\ge
\left(\frac{j-1}{C\,\ee^p\,d^{\,p}}\right)^{1/(p(1+\ln\,d)} 
\quad\mbox{for all}\quad j\ge j_d.   
$$   
Therefore  
$$
\left(1+\tfrac12\,\ln\,\frac{{\rm   
      CRI}_d}{\lambda_{d,j}}\right)^{-\tau(1+\ln\,d)}\le   
C^{\,\tau/p}\,\ee^{\tau}\,d^{\,\tau}\,(j-1)^{-\tau/p}   
\quad\mbox{for all}\quad j\ge j_d.   
$$

Finally,    
$$   
\sum_{j=1}^\infty   
\left[1+\tfrac12\,\ln\,\max\left(1,\frac{{\rm CRI}_d}{\lambda_{d,j}}   
\right)\right]^{-\tau(1+\ln\,d)}\le j_d+   
C^{\,\tau/p}\,\ee^\tau\,d^{\,\tau}\,\sum_{j=j_d+1}^\infty (j-1)^{-\tau/p}.   
$$   
The last series is finite if we take $\tau>p$. Therefore   
$$   
M=\sup_{d\in\NN}\,d^{-\tau}   
\sum_{j=1}^\infty   
\left[1+\tfrac12\,\ln\,\max\left(1,\frac{{\rm CRI}_d}{\lambda_{d,j}}   
\right)\right]^{-\tau(1+\ln\,d)}<\infty,   
$$   
as claimed.    
   
Furthermore, the infimum of $\tau$ satisfying \eqref{eq:EXP-QPT-NOR}    
is at most $p$ and $p$ can be arbitrarily close to the exponent   
of EXP-QPT-ABS/NOR. Hence, $p^*   
=\inf\{\,\tau\,:\ \mbox{$\tau$ satisfies    
\eqref{eq:EXP-QPT-NOR}}\,\}/$.   
This completes the proof.   
\end{proof}

\subsection{$(s,t)$-weak tractability}   
   
\begin{thm}[EXP-$(s,t)$-WT-ABS/NOR]\label{thm:exp-wt-nor}   
\quad   
   
$\calS$ is EXP-$(s,t)$-WT-ABS/NOR iff   
\begin{equation}\label{eq:EXP-WT}   
\mu (c,s,t):=\sup_{d\in\NN} \sigma (c,d,s)    
\exp (-cd^{\,t}) < \infty\quad \forall c>0,   
\end{equation}   
where   
\[   
 \sigma (c,d,s):=\sum_{j=1}^\infty    
\exp \left(-c \left[   
1+   
\ln\left(2\,\max\left(1,\frac{{\rm   
        CRI}_d}{\lambda_{d,j}}\right)\right)   
\right]^s \right).   
\]   
\end{thm}   
\begin{proof}   
First of all, note that \eqref{eq:EXP-WT} combines the formulas in   
Table 4 for EXP-$(s,t)$-WT-ABS/NOR. Indeed, for ABS, we have    
${\rm CRI}_d=1$ and     
$$   
1+\ln\left(2\,\max\left(1,\frac{{\rm   
        CRI}_d}{\lambda_{d,j}}\right)\right)=   
1+\ln\left(2\,\max\left(1,\frac1{\lambda_{d,j}}\right)\right),   
$$   
whereas for NOR, we have ${\rm CRI}_d=\lambda_{d,1}$ and    
${\rm CRI}_d/\lambda_{d,j}\ge1$. This yields   
$$   
1+\ln\left(2\,\max\left(1,\frac{{\rm   
        CRI}_d}{\lambda_{d,j}}\right)\right)=   
1+\ln\left(2\,\frac{\lambda_{d,1}}{\lambda_{d,j}}\right).   
$$   
Let us first assume that \eqref{eq:EXP-WT} holds.    
We then need to show   
\[   
  \lim_{d + \varepsilon^{-1}\rightarrow\infty}\,    
\frac{\ln \max(1, n_{\ABS/\NOR}(\varepsilon, S_d)) }{d^{\,t}   
+ (1+ \ln\,\max(1,\varepsilon^{-1}))^s}=0.   
\]   
The terms in $\sigma (c,d,s)$ are non-increasing, so   
we have    
\[\exp (-c d^{\,t})\, j\,    
\exp\left(-c\left[1+\ln    
\left(2\,\max\left(1,\frac{{\rm CRI}_d}{\lambda_{d,j}}\right)\right)   
\right]^s \right)\,\le\, \mu (c,s,t).   
\]   
Equivalently,   
\[   
 \exp \left(c\left[1+\ln    
\left(2\,\max\left(1,\frac{{\rm CRI}_d}{\lambda_{d,j}}\right)\right)   
\right]^s\right)\ge   
 \frac{j}{\mu (c,s,t) \exp (cd^{\,t})}.   
\]   
In particular, for $j> \mu (c,s,t)\exp (cd^{\,t})$ we obtain   
\[   
 1+\ln\left(2\,\max\left(1,\frac{{\rm CRI}_d}{\lambda_{d,j}}\right)\right)   
\ge \left(\frac{\ln (j / (\mu (c,s,t) \exp (cd^{\,t}))}{c}\right)^{1/s},   
\]   
or, equivalently,   
\[   
\min\left(1, \frac{\lambda_{d,j}}{{\rm CRI}_d}\right)\le    
2\exp\left(1-\left(\frac{\ln (j / (\mu (c,s,t)    
\exp (cd^{\,t}) )}{c}\right)^{1/s}\right).   
\]   
   
Let now $\varepsilon>0$. We have   
\[   
  2\exp\left(1-\left(\frac{\ln (j / (\mu (c,s,t)    
\exp (cd^{\,t}) )}{c}\right)^{1/s}\right)\le \varepsilon^2   
\]   
iff   
\[   
 j\ge \mu (c,s,t)    
\exp\left(c\left(\left[\max\left(0,1+ \ln \frac{2}{\varepsilon^2}\right)   
\right]^s + d^{\,t}\right)\right).   
\]   
Therefore, if   
\begin{equation}\label{eq:WTnlarge}   
 j_{\varepsilon,d}:=\left\lceil \max(1,\mu (c,s,t))    
\exp\left(c\left(\left[\max\left(0,1+ \ln \frac{2}{\varepsilon^2}\right)   
\right]^s + d^{\,t}\right)\right)\right\rceil   
\end{equation}   
we have   
\[   
 \min\left(1,\frac{\lambda_{d,j_{\varepsilon,d}}}{{\rm CRI_d}}\right)   
\le \varepsilon^2.   
\]   
We now estimate $j_{\varepsilon,d}$.    
Since $\max(1,\mu (c,s,t))\ge 1$,    
the argument of the ceiling function in the right-hand side of    
\eqref{eq:WTnlarge} is also    
at least 1 and we can use $\lceil x\rceil\le 2x$ for all $x\ge1$, so that    
$$   
j_{\varepsilon,d}\le 2\max(1,\mu(c,s,t))   
\exp\left(c\left(\left[\max\left(0,1+ \ln \frac{2}{\varepsilon^2}\right)   
\right]^s + d^{\,t}\right)\right).   
$$   
It is easy to check that    
$$   
\max\left(0,1+\ln\,\frac2{\varepsilon^2}\right)\le    
2\left(1+\ln\,\max(1,\varepsilon^{-1})\right)   
\ \ \ \mbox{for all   
$\varepsilon>0$.}   
$$   
Hence,   
$$   
j_{\varepsilon,d}\le 2\max(1,\mu(c,s,t))\,   
\exp\left(2^sc\left(\left(1+\ln\,\max(1,\varepsilon^{-1})   
\right)^s+d^{\,t}\right)\right)   
$$   
which can be abbreviated as   
$$   
j_{\varepsilon,d}=\mathcal{O}\left(\exp\left(2^sc\left(1+   
\ln\,\max(1,\varepsilon^{-1})\right)^s+d^{\,t}\right)\right),   
$$   
where the factor in the big $\mathcal{O}$ notation is independent of   
$\varepsilon^{-1}$ and $d$.    
   
For NOR, we have   
\[   
\min\left(1,\frac{\lambda_{d,j_{\varepsilon,d}}}{{\rm CRI_d}}\right)=   
 \frac{\lambda_{d,j_{\varepsilon,d}}}{{\lambda_{d,1}}}\le \varepsilon^2,   
\]   
and therefore   
$$   
n_{\NOR}(\varepsilon,S_d)\le j_{\varepsilon,d}=   
\mathcal{O}\left(\exp\left(2^s c\,   
\left(1+\ln\,\max(1,\varepsilon^{-1})\right)^s   
+d^{\,t}\right)\right).   
$$    
Since this holds for all $c>0$, we obtain EXP-$(s,t)$-WT-NOR.   
   
For ABS, let   
$$   
j_1^*(d)=|\{\,j\in\NN\,:\ \lambda_{d,j}>1\,\}|.   
$$   
Then    
$$   
\mu(c,s,t)\ge\exp(-cd^{\,t})\,\sum_{j=1}^{j^*_1(d)}   
\exp(-c(1+\ln\,2)^s)=   
\exp\left(-c(d^{\,t}+(1+\ln2)^s)\right)\,j_1^*(d).   
$$   
Hence,   
$$   
j^*_1(d)\le \mu(c,s,t)\,\exp\left(c((1+\ln2)^s+d^{\,t})\right)=   
\mathcal{O}\left(\exp(cd^{\,t})\right),   
$$   
again with the factor in the big $\mathcal{O}$ notation independent of   
$d$.    
{}Note that    
$$   
\max\left(j^*_1(d),j_{\varepsilon,d}\right)   
=\mathcal{O}\left(\exp\left(   
2^sc\left(\left(1+\ln\,\max(1+\varepsilon^{-1})\right)^s+d^{\,t}   
\right)\right)\right).   
$$    
{}For $j=\max\left(j^*_1(d)+1,j_{\varepsilon,d}\right)$    
we have    
\[   
\min\left(1,\frac{\lambda_{d,j}}{{\rm CRI_d}}\right)=   
 \lambda_{d,j}\le \varepsilon^2.   
\]   
Therefore,    
$$   
n_{\ABS}(\varepsilon, S_d)\le j=   
\mathcal{O}\left(\exp(2^sc\,   
(1+\ln\,\max(1,\varepsilon^{-1}))^s   
+d^{\,t})\right).   
$$   
Since this holds for all choices of $c>0$, we obtain EXP-$(s,t)$-WT-ABS.    
   
\vskip 1pc   
Let us now assume that we have EXP-$(s,t)$-WT-ABS/NOR, i.e.,   
\[   
  \lim_{d + \varepsilon^{-1}\rightarrow\infty}\,    
\frac{\ln\max(1,n_{\ABS/\NOR}(\varepsilon, S_d)) }   
{d^{\,t}+ (1+ \ln\,\max(1,\varepsilon^{-1}))^s}=0.   
\]   
Then for any $c>0$ there exists an integer $C=C(c,s,t)$ such that   
\[   
 n:=n_{\ABS/\NOR}(\varepsilon, S_d) \le    
\left\lfloor \exp\left( c\left(   
\left[1+ \ln\,\max(1,\varepsilon^{-1})\right]^s   
+ d^{\,t}\right)  \right)  \right\rfloor   
\]   
for all choices of $\varepsilon^{-1}+d\ge C$.    
   
For $d\in\NN$, choose $\varepsilon>0$ such that $\varepsilon^{-1} \ge    
\max(1,C-d)$.    
Since the eigenvalues $\lambda_{d,j}$ are non-increasing, we have    
\begin{equation}\label{eq:WTlambda}   
 \lambda_{d,\left\lfloor\exp\left( c\left(   
\left[1+ \ln\,\max(1,\varepsilon^{-1})\right]^s   
+ d^{\,t}\right)  \right)  \right\rfloor+1} \le    
\varepsilon^2 {\rm CRI}_d.   
\end{equation}   
Let    
\[   
j=\left\lfloor\exp\left( c\left(   
\left[1+ \ln\,\max(1,\varepsilon^{-1})\right]^s   
+ d^{\,t}\right)  \right)  \right\rfloor+1,   
\]   
and   
\[   
k^*_1(d):=\left\lfloor\exp\left( c\left(   
\left[1+ \ln\,\max(1,C-d))\right]^s   
+ d^{\,t}\right)  \right)  \right\rfloor+1=\Theta\left(\exp(cd^{\,t})\right)   
\]   
for all $d$ with the factor in the $\Theta$ notation   
independent of $d$.   
   
If we vary $\varepsilon^{-1}\in [\max(1,C-d),\infty)$,    
$j$ will attain any integer value greater than or equal to $k_1^*(d)$.    
Furthermore we have   
\[   
 j\le \exp\left( c\left(   
\left[1+ \ln\,\max(1,\varepsilon^{-1})\right]^s   
+ d^{\,t}\right)  \right) +1,   
\]   
or equivalently,   
\[   
\varepsilon   
\le    
\exp\left(- \left(\frac{\ln ((j-1)/\exp (cd^{\,t}))}{c}\right)^{1/s} +1\right)   
\ \ \    
\mbox{for any $j\ge k_1^*(d)$}.   
\]   
Therefore,    
by inserting into \eqref{eq:WTlambda}, we see that    
for all    
$$   
j\ge k_1^{*}(d)=\Theta\left(\exp(cd^{\,t})\right)   
$$   
we have     
\[   
 \frac{\lambda_{d,j}}{{\rm CRI}_d} \le    
\exp\left(- 2\left(\frac{\ln ((j-1)/\exp (cd^{\,t}))}{c}   
\right)^{1/s} +2\right).   
\]   
The latter inequality is equivalent to    
\[   
 c\left[1-\frac{1}{2} \ln \frac{\lambda_{d,j}}{   
{\rm CRI}_d}\right]^s \ge    
 \ln ((j-1)/\exp (cd^{\,t})),   
\]   
which, in turn, is equivalent to    
\[   
 c\left[1+\frac{1}{2} \ln \frac{   
{\rm CRI}_d}{\lambda_{d,j}}\right]^s \ge    
 \ln ((j-1)/\exp (cd^{\,t})).   
\]   
The last inequality holds iff   
\[   
 \exp \left(-2c \left[1+\frac{1}{2} \ln    
\frac{   
{\rm CRI}_d}{\lambda_{d,j}}\right]^s\right)   
\exp (-2cd^{\,t})\le \frac{1}{(j-1)^2}.   
\]   
We are ready to estimate    
$$   
 \exp(-2cd^{\,t})\,\sigma (2c,d,s)=   
\exp(-2cd^{\,t})\,\sum_{j=1}^\infty    
\exp \left(-2c \left[1+\ln\left(2\,\max\left(1,   
\frac{{\rm CRI}_d}{\lambda_{d,j}}\right)\right)\right]^s\right).   
$$   
   
{}For NOR, we have $\max(1,{\rm   
  CRI}_d/\lambda_{d,j})=\lambda_{d,1}/\lambda_{d,j}$ and   
$$   
1+\ln\,\frac{2\lambda_{d,1}}{\lambda_{d,j}}\ge 1+\tfrac12\,\ln\,   
\frac{   
\lambda_{d,1}}{\lambda_{d,j}}.   
$$   
Therefore,   
\begin{eqnarray*}   
 \exp(-2cd^{\,t})\,\sigma (2c,d,s)&\le&   
\exp(-2cd^{\,t})\,   
\sum_{j=1}^\infty \exp \left(-2c \left[1+\frac{1}{2}\ln    
 \frac{   
\lambda_{d,1}}{\lambda_{d,j}}\right]^s \right)\\   
&\le& \exp (-2cd^{\,t}) (k^{*}_1(d)-1) + \sum_{j=k^{*}_1(d)}^\infty    
\frac{1}{(j-1)^2}.   
\end{eqnarray*}   
   
Obviously, the latter sum is bounded by $\pi^2 /6$. Furthermore,   
$$   
 \exp (-2cd^{\,t}) (k^{*}_1(d)-1)=\mathcal{O}\left(\exp(-cd^{\,t})\right).   
$$   
Hence for any $c>0$ it is true that   
$$   
 \mu (2c,s,t)=\sup_{d\in\NN} \sigma (2c,d,s) \exp (-2cd^{\,t})<\infty.   
$$   
By varying the constant $c$, we see    
the validity of \eqref{eq:EXP-WT}, finishing the proof for NOR.   
\vskip 1pc   
For ABS, as before, we consider   
$$   
j_1^*(d)=|\{\,j\ :\ \lambda_{d,j}>1\,\}|.    
$$   
Note that    
$$   
j_1^*(d)\le n_{\ABS}(1,S_d)\ \ \ \mbox{for all $d\in \NN$.}   
$$   
Furthermore, for $d\ge C(c,s,t)$,  with $C(c,s,t)$ defined as before,   
we have    
$$   
n_{\ABS}(1,S_d)\le   
\exp\left(c\left(\left(1+\ln\,2\right)^s+d^{\,t}\right)\right)=   
\mathcal{O}\left(\exp(cd^{\,t})\right).   
$$

We now estimate    
\begin{eqnarray*}   
&\ &\sigma(2c,d,s)=   
\sum_{j=1}^\infty\exp\left(-2c\left[1+\ln\left(2\,\max\left(1,   
\frac1{\lambda_{d,j}}\right)\right)\right]^s\right)\\   
&=&   
\sum_{j=1}^{j_1^*(d)}\exp\left(-2c(1+\ln\,2)^s\right)+   
\sum_{j=j^*_1(d)+1}^\infty   
\exp\left(-2c\left(1+\ln\,\frac{2}{\lambda_{d,j}}\right)^s\right)\\   
&\le&\max(j^*_1(d),k^*_1(d))   
+\sum_{j=\max(j^*_1(d),k^*_1(d))+1}^\infty   
\exp\left(-2c\left(1+\ln\,\frac{2}{\lambda_{d,j}}\right)^s\right).   
\end{eqnarray*}   
   
 Note that for $j\ge \max(j^*_1(d),k_1(d))+1$ we have   
 $$   
\lambda_{d,j}\le1\ \ \  \mbox{and}\ \ \    
1+\ln(2/\lambda_{d,j})\ge 1+   
\tfrac12\ln(1/\lambda_{d,j}).                               
$$    
Therefore, we conclude as before that   
$$   
\sum_{j=\max(j^*_1(d),k^*_1(d))+1}^\infty   
\exp\left(-2c\left(1+\ln\,\frac{2}{\lambda_{d,j}}\right)^s\right)   
\le \frac{\pi^2}6+\exp\left(cd^{\,t}\right).   
$$   
Hence,   
$$   
\exp(-2cd^{\,t})\,\sigma(2c,d,s)=\mathcal{O}\left(   
1+\exp(-2cd^{\,t}+cd^{\,t})\right)   
$$   
is uniformly bounded in $d$, and $\mu(2c,s,t)<\infty$.   
By varying the constant $c$, we conclude the proof for ABS.   
\end{proof}   
\subsection{Uniform weak tractability}   
   
\qquad   
   
We stress that we can verify UWT by checking $(s,t)$-WT for all   
positive $s$ and~$t$ by criteria presented in Table 4. The    
advantage of   
this approach is that these criteria are independent   
of the ordering of the singular    
values $\lambda_{d,j}$'s.    
   
Table 5 presents necessary and sufficient conditions on the decay of   
the ordered eigenvalues $\lambda_{d,n}$'s in order to achieve UWT.    
We need to prove these conditions for both ALG and EXP since    
the case of ALG has also not yet been considered.    
   
\begin{thm}\label{algexpuwt}   
\qquad   
\begin{itemize}   
\item   
$\calS$ is ALG-UWT-ABS/NOR \ iff   
\begin{equation}\label{997}   
\lim_{n\to\infty}\ \inf_{d\le [\ln\,n]^k}\    
\frac{\ln\, \frac{{\rm CRI}_d}{\lambda_{d,n}}}{\ln\,\ln\,n}\ =\ \infty   
\ \ \ \ \mbox{for \ all \ $k\in \NN$.}   
\end{equation}   
\item   
$\calS$ is EXP-UWT-ABS/NOR \ iff   
\begin{equation}\label{998}   
\lim_{n\to\infty}\ \inf_{d\le [\ln\,n]^k}\    
\frac{\ln\left(\max\left(1,\ln\, \frac{{\rm CRI}_d}{\lambda_{d,n}}\right)\right)}{\ln\,\ln\,n}\ =\ \infty   
\ \ \ \mbox{for \ all \ $k\in \NN$.}   
\end{equation}   
\end{itemize}   
\end{thm}   
\begin{proof}   
We first consider ALG. Assume that we have    
ALG-UWT-ABS/NOR. We need to show \eqref{997}.   
Since $\calS$ is ALG-$(s,t)$-WT-ABS/NOR for all positive $s$ and $t$,    
due to Table 4 we have    
for all positive $c$,    
$$   
M_{c,s,t}:=\sup_{d\in\NN}\,\exp(-cd^{\,t})\,\sum_{j=1}^\infty   
\exp\left(-c\left(\frac{{\rm   
        CRI}_d}{\lambda_{d,j}}\right)^{s/2}\right)<\infty.   
$$   
Since the terms $\exp\left(-c({\rm CRI}_d/\lambda_{d,j})^{s/2}\right)$ are   
non-increasing, we obtain   
$$   
\exp(-cd^{\,t})\,n\exp\left(-c\left({\rm CRI}_d/\lambda_{d,n}   
\right)^{s/2}\right)\le M_{c,s,t}.   
$$   
Hence,   
$$   
\exp\left(-c\left(\frac{{\rm CRI}_d}{\lambda_{d,n}}\right)^{s/2}\right)\le   
\frac{M_{c,s,t}\exp(cd^{\,t})}{n},   
$$   
and by taking the logarithms we conclude   
$$   
\left(\frac{{\rm CRI}_d}{\lambda_{d,n}}\right)^{s/2}\ge   
\frac{\ln\,n-\ln(M_{c,s,t})-cd^{\,t}}{c}.   
$$   
Take now an arbitrary (large) integer $k$. For this $k$, we choose   
$t=1/(2k)$. Then there exists $n_{c,s,t}\ge 2$ such that for all    
$n\ge n_{c,s,t}$ and $d\le [\ln\,n]^k$ we have    
$$   
\frac{\ln\,n-\ln(M_{c,s,t})-cd^{\,t}}{c}\ge   
\frac{\ln\,n-\ln(M_{c,s,t})-c(\ln\,n)^{1/2}}{c}\ge (\ln n)^{1/2}.   
$$   
Using this estimate we conclude that   
$$   
\inf_{d\le [\ln\,n]^k}   
\frac{{\rm CRI}_d}{\lambda_{d,n}}\ge (\ln\, n)^{1/s},   
$$   
and by taking again the logarithms   
$$   
\inf_{d\le[\ln\,n]^k}   
\frac{\ln\left(\frac{{\rm CRI}_d}{\lambda_{d,n}}\right)}   
{\ln\,\ln\,n}   
\ge \frac1s\ \ \ \mbox{for all $n\ge n_{c,s,t}$}.   
$$   
Since $s$ can be arbitrarily small, the left hand side of the last   
inequality is arbitrarily large for large $n$. This means that the limit in   
\eqref{997} is infinity, as claimed.     
   
We now assume that \eqref{997} holds. We need to prove   
ALG-$(s,t)$-WT-ABS/NOR for all positive $s$ and $t$.    
Due to Table 4, we need to show that   
$$   
\sup_{d\in\NN}\,   
\exp(-cd^{\,t})\,   
\sum_{j=1}^\infty   
\exp\left(-c\left(\frac{{\rm   
        CRI}_d}{\lambda_{d,j}}\right)^{s/2}\right)<\infty   
\ \ \ \mbox{for all $c>0$}.   
$$   
Take an arbitrary (small) positive $c$. From \eqref{997} we know that   
for all $k\in \NN$ and $M>0$ there exists an integer $N(k,M)\ge 3$ such that   
$$   
\frac{\ln\left(\frac{{\rm CRI}_d}{\lambda_{d,n}}\right)}   
{\ln\,\ln\,n}\ge M\ \ \ \mbox{for all $n\ge N(k,M)$ and    
for all $d\le [\ln\,n]^k$}.   
$$   
Note that $d\le[\ln\,n]^k$ iff $n\ge\exp(d^{\,1/k})$.   
Therefore we can rewrite the last expression as   
$$   
\frac{{\rm CRI}_d}{\lambda_{d,n}}\ge (\ln n)^M\ \ \    
\mbox{for all $d\in \NN$ and $n\ge   
  \max\left(N(k,M),\exp(d^{1/k})\right)$}.   
$$   
Take now $M=4/s$ and $k>1/t$, and let   
$$   
N^*=N(k,M,c,d)=\max\left(N(k,M),\exp(d^{1/k}),\exp(2/c)\right).   
$$   
Then   
$$   
\alpha:=   
\sum_{n=1}^\infty   
\exp\left(-c\left(\frac{{\rm   
        CRI}_d}{\lambda_{d,n}}\right)^{s/2}\right)   
\le N^*-1+\sum_{n=N^*}^\infty\exp\left(-c(\ln\,n)^2\right).   
$$   
Note that $\exp(-c(\ln\,n)^2)\le 1/n^2$ for $n\ge \exp(2/c)$.   
Therefore   
$$   
\alpha\le N^*+\sum_{n=N^*}^\infty\frac1{n^2}\le   
\max\left(N(k,M),\exp(2/c)\right)+\frac{\pi2}{6}+\exp\left(d^{\,1/k}\right).   
$$   
Hence,   
$$   
\exp(-cd^{\,t})\,   
\sum_{n=1}^\infty   
\exp\left(-c\left(\frac{{\rm   
        CRI}_d}{\lambda_{d,n}}\right)^{s/2}\right)   
=\mathcal{O}\left(\exp\left(-cd^{\,t}+d^{\,1/k}\right)\right)   
$$   
with the factor in the big $\mathcal{O}$ notation independent of $d$.   
Since $t>1/k$, the last expression is uniformly bounded in $d$,   
and we have ALG-$(s,t)$-WT-ABS/NOR for all positive $s$ and $t$. This   
means that ALG-UWT-ABS/NOR holds, as claimed.   
\vskip 1pc   
We now consider the case of EXP. Assume first that we have   
EXP-UWT-ABS/NOR. We need to prove \eqref{998}. Since we have    
EXP-$(s,t)$-WT-ABS/NOR for all positive $s$ and $t$,    
due to Theorem \ref{thm:exp-wt-nor} we have for all positive $c$,   
$$   
M_{c,s,t}:=\sup_{d\in\NN}\exp(-cd^{\,t})\,\sum_{j=1}^\infty   
\exp\left(-c\left[1+\ln\left(2\max\left(1,\frac{{\rm   
            CRI}_d}{\lambda_{d,j}}\right)\right)\right]^s\right)<\infty.   
$$   
As for ALG, we conclude that   
$$   
\exp(-cd^{\,t})\,n\,   
\exp\left(-c\left[1+\ln\left(2\max\left(1,\frac{{\rm   
            CRI}_d}{\lambda_{d,n}}\right)\right)\right]^s\right)\le   
M_{c,s,t},   
$$   
which yields   
$$   
\left[1+\ln\left(2\max\left(1,\frac{{\rm   
            CRI}_d}{\lambda_{d,n}}\right)\right)\right]^s   
\ge \frac{\ln\,n-\ln(M_{c,s,t})-cd^{\,t}}{c}   
\ \ \ \mbox{for all $n\in\NN$}.   
$$   
Similarly as before, for an arbitrary integer $k$, we choose   
$t=1/(2k)$   
and conclude the existence of $n_{c,s,t}\ge 3$ such that for all    
$n\ge n_{c,s,t}$   
and all $d\le [\ln\,n]^k$ we have   
$$   
\frac{\ln\,n-\ln(M_{c,s,t})-cd^{\,t}}{c}   
\ge    
\frac{\ln\,n-\ln(M_{c,s,t})-c(\ln\,n)^{1/2}}{c}\ge   
(\ln\,n)^{1/2}.   
$$   
Hence, by taking the logarithms we conclude   
$$   
\ln\left(1+\ln\left(2\max\left(1,\inf_{d\le[\ln\,n]^k}   
\frac{{\rm CRI}_d}{\lambda_{d,n}}   
\right)\right)\right)\ge\frac1{2s}\,\ln\,\ln\, n   
\ \ \ \mbox{for all $n\ge n_{c,s,t}$}.   
$$   
Let    
$$   
x=\inf_{d\le [\ln\,n]^k}\frac{{\rm CRI}_d}{\lambda_{d,n}}.   
$$   
For small $s$ and $n\ge n_{c,s,t}$, we have large $x$, say, at least    
equal to $\exp(2)$. It is easy to check that    
$$   
\ln(1+\ln(2\max(1,x)))\le 2 \ln\left(\max(1,\ln\,x)\right)\ \ \ \mbox{for all $x\ge   
  \exp(2)$}.   
$$   
Therefore, for small $s$ we obtain   
$$   
\inf_{d\le[\ln\,n]^k}\frac{\ln\left(\max\left(1,\ln\,\frac{{\rm   
      CRI}_d}{\lambda_{d,n}}\right)\right)}{\ln\,\ln\,n}   
\ge \frac{1}{4s}.   
$$   
Since $s$ can be arbitrarily small, the limit of the left hand side    
is infinity as $n$ goes to infinity, and \eqref{998} holds.   
\vskip 1pc   
We finally assume that \eqref{998} holds. We need to prove   
EXP-UWT-ABS/NOR,   
or equivalently that EXP-$(s,t)$-WT-ABS/NOR holds for all positive $s$   
and $t$. This means that we need to prove that for all   
positive $c$,   
$$   
M_{c,s,t}:=\sup_{d\in\NN}\ \exp(-cd^{\,t})\,   
\sum_{j=1}^\infty   
\exp\left(-c\left[1+\ln\left(2\max\left(1,\frac{{\rm   
            CRI}_d}{\lambda_{d,j}}\right)\right)\right]^s\right)<\infty.   
$$   
From \eqref{998} we know that for all $k\in\NN$ and $M>0$ there exists    
$N(k,M)\ge3$ such that   
$$   
\frac{{\rm CRI}_d}{\lambda_{d,n}}\ge \exp\left(\left(\ln\,n   
  \right)^M\right)   
$$   
for all $d\in \NN$ and for all $n\ge   
  n^*:=\max\left(N(k,M),\exp(d^{\,1/k}),\exp(2/c)\right)$.   
Let   
$$   
\alpha:=   
 \sum_{j=1}^\infty   
\exp\left(-c\left[1+\ln\left(2\max\left(1,\frac{{\rm   
            CRI}_d}{\lambda_{d,j}}\right)\right)\right]^s\right).   
$$   
Then   
$$   
\alpha\le n^*-1+\sum_{n=n^*}^\infty   
\exp\left(-c(\ln\,n)^{Ms}\right).   
$$   
We now take $M=2/s$ and use again the fact that   
$\exp(-c(\ln\,n)^2)\le 1/n^2$ for $n\ge \exp(2/c)$. Then   
$$   
\alpha\le n^* +\frac{\pi^2}6=\mathcal{O}\left(\exp(d^{\,1/k})\right).   
$$   
Taking $k>1/t$, we conclude that   
$$   
M_{c,s,t}=\sup_{d\in\NN}\ \mathcal{O}   
\left(\exp\left(-cd^{\,t}+d^{\,1/k}\right)\right)<\infty.   
$$   
This completes the proof.   
\end{proof}   
   
\section*{Acknowledgements}   
P.~Kritzer is supported by the Austrian Science    
Fund (FWF): Project F5506-N26,    
which is part of the Special Research    
Program ``Quasi-Monte Carlo Methods: Theory and Applications''.   
   
H.~Wo\'zniakowski is supported in part by the National Science Center,   
Poland, based on the decision DEC-2017/25/B/ST1/00945.

The authors gratefully acknowledge the partial support    
of the Erwin Schr\"odinger International Institute    
for Mathematics and Physics (ESI) in Vienna under    
the thematic programme ``Tractability of High Dimensional    
Problems and Discrepancy''    
and the partial support of the National Science Foundation (NSF)    
under Grant DMS-1638521 to the    
Statistical and Applied Mathematical Sciences Institute (SAMSI).   
The research started during the stay of the authors in Vienna    
and during the conference at SAMSI.

\bigskip

\begin{small}   
\noindent\textbf{Authors' addresses:}\\   
   
 \noindent Peter Kritzer\\   
 Johann Radon Institute for Computational and Applied Mathematics (RICAM)\\   
 Austrian Academy of Sciences\\   
 Altenbergerstr. 69, 4040 Linz, Austria.\\   
 \texttt{peter.kritzer@oeaw.ac.at}   
    
 \medskip   
    
 \noindent Henryk Wo\'{z}niakowski\\   
 Department of Computer Science\\   
 Columbia University\\   
 New York 10027, USA.\\   
 Institute of Applied Mathematics\\   
 University of Warsaw\\   
 ul. Banacha 2, 02-097 Warszawa, Poland\\   
 \texttt{henryk@cs.columbia.edu}

\end{small}


\begin{thebibliography}{00}   
   
\bibitem{GW11} M.~Gnewuch, H.~Wo\'{z}niakowski:    
Quasi-polynomial tractability. J. Complexity 27, 312--330, 2011.   
   
   
\bibitem{NW08} E.~Novak, H.~Wo\'zniakowski: \textit{Tractability    
of Multivariate Problems, Volume I: Linear Information}. EMS, Z\"urich, 2008.   
   
\bibitem{NW10}   
E.~Novak, H.~Wo\'zniakowski: \textit{Tractability    
of Multivariate Problems, Volume II: Standard Information    
for Functionals}. EMS, Z\"urich, 2010.   
   
\bibitem{NW12}   
E.~Novak, H.~Wo\'zniakowski: \textit{Tractability    
of Multivariate Problems, Volume III:    
Standard Information for Operators}. EMS, Z\"urich, 2012.   
   
   
\bibitem{S13}   
P. Siedlecki: \textit{Uniform weak tractability}.   
J. Complexity, 29: 438--453, 2013.   
   
\bibitem{TWW88}   
J. F.~Traub, G.W.~Wasilkowski, H. Wo\'zniakowski:   
\textit{Information-Based Complexity}, Academic Press, New York, 1988.   
   
   
\bibitem{WW17} A.~Werschulz, H.~Wo\'{z}niakowski: A new    
characterization of $(s,t)$-weak tractability. J. Complexity 38: 68--79, 2017.   
   
   
\end{thebibliography}
\end{document}